\documentclass[a4paper,12pt]{article}

\usepackage{color}
\usepackage[pdftex]{graphicx}
\usepackage{amsmath,amssymb,amsthm}
\usepackage{wrapfig}
\usepackage{ifpdf}

\newtheorem{theorem}{Theorem}[section]
\newtheorem{cor}{Corollary}[section]

\newtheorem{df}{Definition}[section]

\newcommand{\conv}{\mathop{\rm conv}\nolimits}

\newcommand{\vrt}{\mathop{\rm Vert}\nolimits}
\newcommand{\st}{\mathop{\rm St}\nolimits}

\newcommand{\ep}{\mathop{\rm EP}\nolimits}

\title {KKM type theorems with boundary conditions}

\author {Oleg R. Musin\thanks{This research is partially supported by NSF grant DMS - 1400876.}}

\begin{document}

	\ifpdf \DeclareGraphicsExtensions{.pdf, .jpg, .tif, .mps} \else
	\DeclareGraphicsExtensions{.eps, .jpg, .mps} \fi	
	
\date{}
\maketitle

\begin{abstract}  We consider generalizations of  Gale's colored KKM lemma and Shapley's KKMS theorem. It is shown that spaces and covers can be much more general and the boundary KKM rules can be substituted by more weaker boundary assumptions.   
\end{abstract}

\medskip


\section{Introduction}

Sperner's lemma \cite{Sperner} is a combinatorial analog of the Brouwer fixed point theorem, which is equivalent to it. This lemma and its extension for covers, the KKM (Knaster --Kuratowski -- Mazurkiewicz) theorem \cite{KKM}, have many applications in combinatorics, algorithms, game theory and mathematical economics. 

There are many extensions of the KKM theorem. In this paper we consider two of them: Gale's  lemma \cite{Gale} and  Shapley's KKMS theorem \cite{Sh}. 

 David Gale  in \cite{Gale} proved an existence theorem for an exchange equilibrium in an economy with indivisible goods and only one perfectly divisible good, which can be thought of as money. The main lemma for this theorem is  \cite[Lemma, p. 63]{Gale}. 
 
  Gale's lemma can be considered as a {\it colored  KKM} theorem.  In \cite[p. 63]{Gale}, 
 Gale wrote about his lemma: ``A colloquial statement of this result is the {\it red, white and blue lemma} which asserts that {\it if each of three people paint a triangle red, white and blue according to the KKM rules, then there will be a point which is in the red set of one person, the white set of another, the blue of the third.}''  Note that Bapat \cite{Bapat} found an analog of Gale's lemma for Sperner's lemma. 
 
 In fact, Gale's lemma (or its discrete analogs) can be applied for  fair division problems, namely for  the envy--free cake--cutting and rental harmony problems . In the envy--free {\it cake--cutting problem,} a ``cake'' (a heterogeneous divisible resource) has to be divided among $n$ partners with different preferences over parts of the cake \cite{DS,Stn,Strom,Su}. The cake has to be divided into $n$ pieces such that: (a) each partner receives a single connected piece, and (b) each partner believes that his/her piece is (weakly) better than all other pieces. An algorithm for solving this problem was developed by Forest Simmons in 1980, in a correspondence with Michael Starbird. It was first publicized by Francis Su in 1999 \cite{Su}. 

Suppose a group of friends consider renting a house but they shall first agree on how to allocate its rooms and share the rent. They will rent the house only if they can find a room assignment-rent division which appeals to each of them.
Following Su \cite{Su}, we call such a situation {\it rental harmony.} In \cite{AZ} consideration is given to different aspects of this model.

In 1967 Scarf \cite{Scarf} proved that any non-transferable utility game whose characteristic function is balanced, has a non--empty core. His proof is based on an algorithm which approximates fixed points. Lloyd Shapley \cite{Sh}  replaced the Scarf algorithm by a covering theorem (the KKMS theorem) being a generalization of the KKM theorem. Now Shapley's KKMS theorem \cite{Her,Kom,Sh,SV} is an important tool in the general equilibrium theory of economic analysis.

The main goal of this paper to considere generalizations of Gale's and Shapley's KKMS theorems with general boundary conditions. In our paper \cite{MusH}  with any cover of a space $T$ we associate certain homotopy classes of maps from $T$ to $n$--spheres. 
These homotopy invariants can then be considered as obstructions for extending covers of a subspace $A \subset X$ to a cover of all of $X$. We are using these obstructions to obtain generalizations of the  KKM  and Sperner lemmas.  In particular, we show that in the case when $A$ is a $k$--sphere and $X$ is a $(k+1)$--disk there exist KKM type lemmas for covers by $n+2$ sets if and only if the homotopy group $\pi_{k}({\Bbb S}^{n})\ne0$. In Section 2 is given a review of main results of \cite{MusH}. 

In Section 3 we generalize Gale's lemma. In particular, see Corollary 3.1, we pove that {\it if each of $n$ people paint a $k$--simplex with $n$ colors such that the union of these covers is not null-homotopic on the boundary,  then there will be a point which is in the first color set of one person, the second color set of another, and so on.}

In Section 4 we consider KKMS type theorems. Actually, these theorems are analogs for covers of a polytopal type Sperner's lemmas, see \cite{  DeLPS, MusSpT, MusH}.  Let $V$ be a set of $m$ points in ${\Bbb R}^n$. Then, see Corollary 4.1, {\it if   $\mathcal F = \{F_1,\ldots, F_m\}$ is a cover of a $k$--simplex that is not null--homotopic on the boundary, then there is a balanced with respect to $V$ subset $\mathcal B$ in $\{1,\ldots,m\}$ such that all the $F_i$, for $i\in \mathcal B$, have a common point.}  

If $V$ is the set of vertices of a $k$--simplex $\Delta^k$, then this corollary implies the KKM theorem and if $V$ is the set of all centers of $\Delta^k$ it yields  the KKMS theorem. As an example, we consider  a generalization of Tucker's lemma (Corollary 4.3). (Note that David Gale, Lloyd Shapley as well as John F. Nash were Ph.D. students of Albert W. Tucker in Princeton.)

\medskip 

\noindent{\bf Notations.}  Throughout this paper we consider only normal topological spaces, all simplicial complexes be finite, all manifolds be compact and piecewise linear, $I_n$ denotes the set $\{1,\ldots,n\}$,  $\Delta^n$ denotes the  $n$--dimensional simplex, ${\Bbb S}^{n}$ denotes the $n$--dimensional unit sphere, ${\Bbb B}^{n}$ denotes the $n$--dimensional unit disk and $|K|$ denotes the underlying space of a simplicial complex $K$. 
We shall denote the set of homotopy classes of continuous maps from X to Y as $[X,Y]$.

\section{Sperner -- KKM lemma with boundary conditions}

\medskip

The $(n-1)$--dimensional unit simplex $\Delta^{n-1}$ is defined by 
$$
\Delta^{n-1}:=\{x\in{\Bbb R}^n\, |\, x_i\ge0, x_1+\ldots+x_n=1\}. 
$$ 
Let $v_i:=(x_1,\ldots,x_n)$ with $x_i=1$ and $x_j=0$ for $j\ne i$. Then $v_1,\ldots, v_n$ is  the set of vertices of $\Delta^{n-1}$ in ${\Bbb R}^{n}$.

Let $K$ be a simplicial complex. Denote by $\vrt(K)$ the vertex set of $K$. An $n$--labeling $L$ is a map $L:\vrt(K)\to\{1,2,\ldots,n\}$.  
 Setting 
$$
f_{L}(u):=v_k, \; \mbox{ where } \; u\in\vrt(K) \mbox{ and } k=L(u), 
$$
we have a map $f_{L}:\vrt(K)\to{\Bbb R}^{n}$. 
Every point $p\in |K|$ belongs to the interior of exactly one simplex in $K$. Letting $\sigma=\conv\{u_0,u_1,\ldots,u_k\}$ be the simplex, we have $p=\sum_0^k\lambda_i(p)u_i$ with $\sum_0^k\lambda_i(p)=1$ and all $\lambda_i>0.$ (Actually, $\lambda_i(p)$ are the barycentric coordinates of $p$.) Then $f_{L}$ can be extended to a continuous (piecewise linear) map $f_{L}:|K|\to\Delta^{n-1}\subset{\Bbb R}^{n}$ defined by 
$$
f_{L}(p) = \sum\limits_{i=0}^k{\lambda_i(p)f_{L}(u_i)}. 
$$

We say that a simplex $s$ in $K$ is {\it fully labeled} if $s$ is labeled with a complete set of labels $\{1,2,\ldots,n\}$. Suppose there are no fully labeled  simplices in $K$. Then $f_L(p)$ lies in the boundary of $\Delta^{n-1}$. Since the boundary $\partial\Delta^{n-1}$ is homeomorphic to the sphere ${\Bbb S}^{n-2}$, we have a continuous map 
 $f_{L}:|K|\to {\Bbb S}^{n-2}$. Denote the homotopy class $[f_{L}]\in[|K|,{\Bbb S}^{n-2}]$ by $\mu(L)=\mu(L,K)$. 

\medskip

\noindent {\bf Example 2.1.} 
Let $L:\vrt(K)\to\{1,2,3\}$ be a labelling of a closed planar polygonal line $K$ with vertices $p_1p_2\ldots p_k$. 
Then $f_{L}$ is map from $|K|={\Bbb S}^1$ to ${\Bbb S}^1$. Not that $\mu(L)\in[{\Bbb S}^1,{\Bbb S}^1]={\Bbb Z}$ is the degree of the map $f_{L}$.  Moreover, $$\mu(L)=\deg(f_{L}):=p_*-n_*,$$ 
where $p_*$ (respectively, $n_*$) is the number of (ordering) pairs $(p_i,p_{i+1})$ such that $L(p_i)=1$ and $L(p_{i+1})=2$ (respectively, $L(p_i)=2$ and $L(p_{i+1})=1$). It is clear, that instead of $[1,2]$ we can take $[2,3]$ or $[3,1]$.

For instance, let $L=(1221231232112231231)$. Then $p_*=5$ and $n_*=2$. Thus, $$\mu(L)=5-2=3.$$

\medskip 

\noindent {\bf Example 2.2.} Let $K$ be a triangulation of the boundary of a simplex $\Delta^{k+1}$. In other words, $K$ is a triangulation of of ${\Bbb S}^k$). Let $L:\vrt(K)\to\{1,\ldots,n\}$ be a labeling such that $K$ has no simplices with $n$ distinct labels. Then $f_L\in\pi_k({\Bbb S}^{n-2})$. 

In the case $k=n-2$ we have $\pi_k({\Bbb S}^{k})={\Bbb Z}$ and 
$$[f_L]=\deg(f_L)\in{\Bbb Z}.$$
(Here by $\deg(f)$ is denoted the degree of a continuous map $f$ from  ${\Bbb S}^k$  to itself.)

For instance, let $L$ be a {\it Sperner labeling} of a triangulation $K$ of $\partial\Delta^{n-1}=u_1\ldots u_n$. The rules of this labeling are:\\
(i) The vertices of $\Delta^{n-1}$ are colored with different colors, i. e. $L(u_i)=i$ for $1\le i \le n$.\\
(ii)   Vertices of $K$ located on any $m$-dimensional subface of the large simplex
$u_{i_0}u_{i_1}\ldots u_{i_{m}}$ are colored only with the colors $i_0,i_1,\ldots,i_{m}.$ 
\\ Then $\mu(L)=\deg(f_L)=1$ in $[{\Bbb S}^{n-2},{\Bbb S}^{n-2}]={\Bbb Z}$. 

\medskip

In \cite{MusH} we proved the following theorem, see \cite[Corollary 3.1]{MusH}. 
\begin{theorem} \label{SpT1}  Let $T$ be a triangulation of a simplex $\Delta^{k+1}$.  Let $L:\vrt(T)\to\{1,\ldots,n\}$ be a labeling such that   $T$ has no simplices on the boundary with $n$ distinct labels. If $\mu(L,\partial T)\ne 0$, then $T$ must contain a fully labeled  simplex. 

(Here by $\mu(L,\partial T)$ we denote the invariant $\mu$ on the boundary of $T$.) 
\end{theorem}

Since for a Sperner labeling $\mu(L,\partial T)=1\ne0$, Theorem \ref{SpT1} implies: 

\medskip 

\noindent{\bf (Sperner's lemma \cite{Sperner})} 
	{\it Every Sperner labeling of a triangulation of $\Delta^{n-1}$ contains a cell labeled with a complete set of labels: $\{1,2,\ldots, n\}$.}
	
\medskip

Consider  an oriented manifold $M$ of dimension $(n-1)$ with boundary. 
Then $[\partial M,{\Bbb S}^{n-2}]={\Bbb Z}$ and for any continuous $f:\partial M\to{\Bbb S}^{n-2}$ we have $[f]=\deg{f}$.  If $T$ is a triangulation of $M$ and $L:\vrt(T)\to\{1,\ldots,n\}$ is a labeling  then we denote by $\deg(L,\partial T)$ the class $\mu(L,\partial T)$. 

\begin{theorem} \cite[Theorem 3.4]{MusH}
\label{th34} 
Let $T$ be a triangulation of an oriented manifold $M$ of dimension $(n-1)$ with boundary. 
Then for a labeling $L:\vrt(T)\to \{1,\ldots,n\}$ the triangulation  must  contain at least $|\deg(L,\partial T)|$ fully labelled simplices. 
\end{theorem}

In Fig.1 is shown an illustration of Theorem \ref{th34}. We have a labeling with  $\deg(L,\partial T)=3$. Therefore, the theorem garantee that there are at least three fully labeled triangles. 

\medskip 

\begin{figure}
\begin{center}

  \includegraphics[clip,scale=0.9]{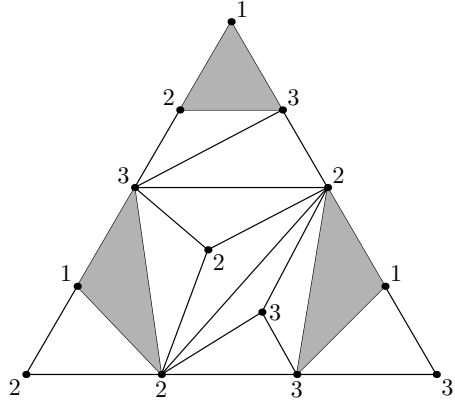}
\end{center}
\caption{$\deg(L,\partial T)=3$. There are three fully labeled triangles.}
\end{figure}

\medskip 

Actually, a labeling can be considered as a particular case of a covering.  For any labeling $L$ there is a natural open cover of $|K|$. 
The open star of a vertex $u\in\vrt(K)$ (denoted $\st(u)$) is $|S|\setminus |B|$, where $S$  is the set of all simplices in $K$ that contain $u$, and $B$ is the set of all simplices in $S$ that contain no $u$.  Let 
 $$ \mathcal U_L(K)=\{U_1(K),\ldots,U_n(K)\},$$  where 
$$
U_\ell(K):=\bigcup\limits_{u\in W_\ell}{\st(u)}, \; \; W_\ell:=\{u\in\vrt(K): L(u)=\ell\}. 
$$
It is clear, that $K=U_1(K)\bigcup\ldots\bigcup{U_{n}(K)}$, i. e. $\mathcal U_L(K)$ is a cover of $K$. 

\medskip

Now we extend the definition of $\mu(L)$ for covers. 

Let $\mathcal U=\{U_1,\ldots,U_n\}$ be  a collection of open sets whose union contains a space $T$. In other words, $\mathcal U$ is a cover of $T$.  Let  $\Phi=\{\varphi_1,\ldots,\varphi_n\}$ be a partition of unity subordinate to $\mathcal U$, i. e. $\Phi$ is a collection of non--negative functions on $T$ such that supp$(\varphi_i)\subset U_i, \, i=1,\ldots,n,$  and  for all $x\in T$, $\sum_1^n{\varphi_i(x)}=1$. 
Let 
$$
f_{\mathcal U,\Phi}(x):=\sum\limits_{i=1}^n{\varphi_i(x)v_i},
$$
where $v_1,\ldots,v_n$, as above, are vertices of $\Delta^{n-1}$.

Suppose the intersection of all $U_i$ is empty.   
Then $f_{\mathcal U,\Phi}$ is a continuous map from $T$ to ${\Bbb S}^{n-2}$. 

In \cite[Lemmas 2.1 and 2.2]{MusH} we proved that a homotopy class $[f_{\mathcal U,\Phi}]$ in $[T,{\Bbb S}^{n-2}]$ does not depend on $\Phi$. We denote it by $\mu(\mathcal U)$.

Note that for a labeling $L:\vrt(K)\to\{1,2,\ldots,n\}$ we have 
$$
\mu(L)=\mu(\mathcal U_L(K))\in [|K|,{\Bbb S}^{n-2}].
$$

\medskip 

\noindent {\bf Example 2.3.} (a)  Let $T={\Bbb S}^k$ and $\mathcal U=\{U_1,\ldots,U_n\}$ be an open cover of $T={\Bbb S}^k$ such that the intersection of all $U_i$ is empty. Then $\mu(\mathcal U)\in\pi_k({\Bbb S}^{n-2})$. In the case $k=n-2$ we have 
$$\mu(\mathcal U)=\deg(f_{\mathcal U})\in{\Bbb Z}.$$

(b) Let $h:{\Bbb S}^k\to{\Bbb S}^{n-2}$ be any continuous map. Actually, ${\Bbb S}^{n-2}$ can be considered as the boundary of the simplex $\Delta^{n-1}$. 
Let 
$$
U_i:=h^{-1}(U_i(\Delta^{n-1}), \; i=1,\ldots,n \, \; \mathcal U=\{U_1,\ldots,U_{n}\}. 
$$
Then $\mu(\mathcal U)=[h]\in \pi_k({\Bbb S}^{m-1})$.

For instance, if $h:{\Bbb S}^3\to{\Bbb S}^2$ is the Hopf fibration, then 
 $\mu(\mathcal U)=1\in \pi_3({\Bbb S}^2)={\Bbb Z}$. 

\medskip

 In fact, see  \cite[Lemma 2.4]{MusH}, the homotopy classes of covers are also well defined for closed sets. 

\begin{df}
  We call a family of sets $\mathcal S=\{S_1,\ldots,S_n\}$  as a cover of a space $T$ if $\mathcal S$ is either an open or closed cover of $T$.
\end{df}

\begin{df}  Let $\mathcal S=\{S_1,\ldots,S_n\}$   be a cover of a space $T$. We say that $\mathcal S$ is not null--homotopic if   the intersection of all $S_i$ is empty  and $\mu(\mathcal S)\ne0$ in $[T,{\Bbb S}^{n-2}]$. 
\end{df}

Note that covers in  Examples 2.1, 2.2 and 2.3 are not null--homotopic. Actually, these examples and \cite[Theorems 2.1--2.3]{MusH} imply the following theorem. 

\begin{theorem} \label{th21} Let $\mathcal S=\{S_1,\ldots,S_m\}$   be a cover of a space $T$. Suppose the intersection of all $S_i$ is empty. 
\begin{enumerate}
	\item Let $T$ be ${\Bbb S}^k$. Then   $\mathcal S$ is not null--homotopic if  and only if $\pi_k({\Bbb S}^{n-1})$ is not trivial and $\mu({\mathcal S})\ne0$  in this group.  
	\item Let $T$ be an oriented $(n-2)$--dimensional manifold. Then   $\mathcal S$ is not null--homotopic if  and only if $\deg(f_{\mathcal S})\ne0$. 
	\end{enumerate}	
\end{theorem}


Now we consider a generalization of the KKM theorem for covers of spaces. 

\begin{df}	
Consider a pair $(X,A)$, where $A$ is a subspace of a space $X$. Let  $\mathcal S=\{S_1,\ldots,S_m\}$ be a cover of $X$ and $\mathcal C=\{C_1,\ldots,C_m\}$ be a cover of $A$. 
We say that $\mathcal S$ is an extension of $\mathcal C$ and write  $ \mathcal C=\mathcal S|_A$ if  $C_i=S_i\cap A$ for all $i$. 
\end{df}

\begin{df}
We say that a pair of spaces $(X,A)$, where $A\subset X$, belongs to $\ep_n$ and write $(X,A)\in\ep_n$ if   there is a continuous map $f:A\to {\Bbb S}^n$  with $[f]\ne0$ in $[A,{\Bbb S}^n]$ that cannot be extended to a continuous map $F:X\to {\Bbb S}^n$ with $F|_A=f$. 
\end{df}

We denoted this class of pairs by $\ep$ after S. Eilenberg and L. S. Pontryagin  who initiated obstruction theory in the late 1930s.    Note that \cite[Theorem 2.3]{MusH} yield that 
\\(i) if $\pi_k({\Bbb S}^n)\ne0$, then $({\Bbb B}^{k+1},{\Bbb S}^k)\in\ep_n$, \\
(ii) if $X$ is an oriented $(n+1)$--dimensional manifold and $A=\partial X$, then $(X,A)\in\ep_n$. 

\begin{theorem} \label{thKKM1}  Let $A$ be a subspace of a space $X$. Let  $(X,A)\in\ep_{n-2}$.  
Let  $\mathcal S=\{S_1,\ldots,S_n\}$  be  a cover of $(X,A)$. Suppose the cover $\mathcal S$ is not null--homotopic on $A$.  Then  all the $S_i$ have a common intersection point. 
\end{theorem}

\begin{cor}  \label{KKM1}   Let  $\mathcal S=\{S_1,\ldots,S_n\}$  be  a cover of  ${\Bbb B}^k$. Suppose  $\mathcal S$ is not null--homotopic on the boundary of  ${\Bbb B}^k$.  Then  all the $S_i$ have a common intersection point. 
\end{cor}

\begin{cor} \label{cor22}  Let  $\mathcal S=\{S_1,\ldots,S_n\}$  be  a cover   of a manifold $M$ of dimension $(n-1)$ with boundary. Let 
 $ \mathcal C:=\mathcal S|_{\partial M}. $
Suppose $\deg{f_{\mathcal C}}\ne 0$.     Then  all the $S_i$ have a common intersection point. 
\end{cor}

 We say that a cover  $\mathcal S:=\{S_1,\ldots,S_m\}$  of a simplex $\Delta^{m-1}$  is a {\it KKM cover} if for all $J \subset I_m$ the face of $\Delta^{m-1}$ that is spanned by vertices $v_i$ for $i \in J$ is covered by $S_i$ for $i \in J$.
 
 Let $\mathcal C:=\mathcal S|_{\partial{\Delta^{n-1}}}.$ Note that if all the $C_i$ have no a common intersection point, then the KKM assumption implies that $\mu(\mathcal C)=\deg(f_{\mathcal C})=1$.    Thus,  Corollary \ref{cor22} yields 

\begin{cor} (KKM theorem \cite{KKM})
If   $\mathcal S:=\{S_1,\ldots,S_n\}$ is a KKM cover of $\Delta^{n-1}$, then all the $S_i$ have a common intersection point.	
\end{cor}


\section{Gale's lemma with boundary conditions}

David Gale \cite[Lemma, p. 63]{Gale} proved the following lemma:  

\medskip

{\it For $i,j=1,2,\ldots,n$ \, let \, $S_j^i$\,  be closed sets such that for each $i$, $\{S_1^i,\ldots,S^i_n\}$
is a KKM covering of $\Delta^{n-1}$. Then there exists a permutation $\pi$ of \, $1,2,\ldots,n$ \, such that
$$
\bigcap\limits_{i=1}^n{S^i_{\pi(i)}}\ne\emptyset. 
$$
}

\medskip






Now we generalize this lemma for pairs of spaces.

\begin{theorem} \label{th4}  

Let $A$ be a subspace of a space $X$. Let  $(X,A)\in\ep_{n-2}$. 
Let  $\mathcal S^i=\{S^i_1,\ldots,S^i_n\}$, $i=1,\ldots,n$,  be covers of $(X,A)$. Let 
$$
 F_j:=\bigcup\limits_{i=1}^n{S_j^i}, \quad \mathcal F:=\{F_1,\ldots,F_n\},  \quad  \mathcal C:=\mathcal F|_A. 
$$
Suppose $\mathcal C$ is not null--homotopic.  Then there exists a permutation $\pi$ of \, $1,2,\ldots,n$ \, such that
$$
\bigcap\limits_{i=1}^n{S^i_{\pi(i)}}\ne\emptyset. 
$$
\end{theorem}
\begin{proof} Here we use Gale's proof of his lemma. 
We consider the case where the sets $S^i_j$ are open. As above the proof of the closed case then follows by a routine limiting argument. Now for each cover $\mathcal S^i$ consider the corresponding partition of unity $\{\varphi_1^i,\ldots,\varphi_n^i\}$. 

  Define $\Phi^i:X\to \Delta^{n-1}$  and $\Phi:X\to \Delta^{n-1}$ by 
$$
\Phi^i(p):=(\varphi^i_1(p),\ldots,\varphi^i_n(p)),  \quad  \Phi(p):=\frac{\Phi^1(p)+\ldots+\Phi^n(p)}{n},
$$
where $\Delta^{n-1}=\{(x_1,\ldots,x_n) \in {\Bbb R}^n \, |\, x_i\ge0,\,  i=1,\ldots,n, \, x_1+\ldots+x_n=1\}$. 
Since  $\mathcal C$ is not null--homotopic,  $\Phi$ is a map from $A$ to $\partial\Delta^{n-1}\equiv{\Bbb S}^{n-2}$ and $\Phi(X)=\Delta^{n-1}$. Therefore, there is $p\in X$ such that $\Phi(p)=(1/n,\ldots,1/n)$, so $n\Phi(p)=(1,\ldots,1)$. Thus, the matrix $M:=n\Phi(p)=\left(\varphi^i_j(p)\right)$ is  a {\it doubly stochastic} matrix, is a square matrix  of nonnegative real numbers, each of whose rows and columns sums to 1. By Mirsky's lemma \cite{Mirsky} for any doubly stochastic matrix $M$ it is a possible to find a permutation $\pi$ such that $\varphi^i_{\pi(i)}(p)>0$ for all $i$, but this means precisely that $p\in S^i_{\pi(i)}$.  
\end{proof}

\noindent{\bf Remark.}. This proof is constructive. If $p\in \Phi^{-1}(1/n,\ldots, 1/n)$,  then there is a permutation $\pi$ with $\varphi^i_{\pi(i)}(p)>0$, i.e. the intersection of all the $S^i_{\pi(i)}$, $i=1,\ldots,n$, is not empty. 

\medskip

Theorem \ref{th4} and \cite[Theorem 2.3]{MusH} imply the following corollaries. 

\begin{cor} Let  $\mathcal S^i=\{S^i_1,\ldots,S^i_n\}$, $i=1,\ldots,n$,  be covers of  ${\Bbb S}^k$. Let  $F_j$ is the union of $S^i_j$, $i=1,\ldots,n$, and 
$\mathcal F:=\{F_1,\ldots,F_n\}. $
Suppose $\mathcal F$ is not null--homotopic on the boundary  of ${\Bbb S}^k$.  Then there exists a permutation $\pi$ of \, $1,2,\ldots,n$ \, such that
$$
\bigcap\limits_{i=1}^n{S^i_{\pi(i)}}\ne\emptyset. 
$$
\end{cor}

\begin{cor} \label{corM} Let  $\mathcal S^i=\{S^i_1,\ldots,S^i_n\}$, $i=1,\ldots,n$,  be covers of a manifold $M$ of dimension $(n-1)$ with boundary. Let  $F_j$ is the union of $S^i_j$, $i=1,\ldots,n$,
$\mathcal F:=\{F_1,\ldots,F_n\} $, and $ \mathcal C:=\mathcal F|_{\partial M}. $
Suppose $\deg{\mathcal C}\ne 0$.   Then there exists a permutation $\pi$ of \, $1,2,\ldots,n$ \, such that
$$
\bigcap\limits_{i=1}^n{S^i_{\pi(i)}}\ne\emptyset. 
$$
\end{cor}

Now consider the rental harmony problem. Following Su \cite{Su}, suppose there are $n$ housemates, and $n$ rooms to assign, numbered $1,\ldots,n$.  Let $x_i$ denote the price of the $i$--th room, and suppose that the total rent is 1. Then $x_1+\ldots+x_n=1$ and $x_i\ge 0$. 
 From this we see that the set of all pricing schemes  forms a  simplex $\Delta^{n-1}$. 
 
Denote by $S^i_j$ a set of price vectors $p$ in $\Delta^{n-1}$ such that housemate $i$ likes room $j$ at these prices.  Consider the following conditions: 

\medskip 

(C1) In any partition of the rent, each person finds some room acceptable. In other words, $\mathcal S^i=\{S^i_1,\ldots,S^i_n\}$, $i=1,\ldots,n$, is a (closed or open) cover of  $\Delta^{n-1}$. 

\medskip

(C2) Each person always prefers a free room (one that costs no rent) to a non--free room. In other words, for all $i$ and $j$, $S^i_j$ contains 
$\Delta^{n-1}_j:=\{(x_1,\ldots,x_n)\in \Delta^{n-1}\,  |\, x_j=0\}$. 

\medskip 

In fact, (C2) is the ``dual'' boundary KKM condition.  Su  \cite[Sect. 7]{Su} using the ``dual'' simplex $\Delta^*$ and  ``dual'' Sperner lemma proves that  there exists a permutation $\pi$ of \, $1,2,\ldots,n$, such that the intersection of the $S^i_{\pi(i)}$, $i=1,\ldots,n$, is not empty. It proves the rental harmony theorem. Also,  this theorem can be derived from Corollary \ref{corM}. 

\medskip 

\noindent{\bf Rental Harmony Theorem} \cite{Su}. {\it Suppose $n$ housemates in an $n$--bedroom house seek to decide who gets which room and for what part of the total rent. Also, suppose that the  conditions (C1) and (C2) hold. Then there exists a partition of the rent so that each person prefers a different room.}

\medskip 

Let us consider an extension of this theorem. Suppose there are some constraints $f_i(p)\le0, \, i=1,\ldots,n,$ for price vectors. Let  $M:=\{p\in \Delta^{n-1}\, |\, f_1(p)\le 0,\ldots, f_k(p)\le0\}$ be a manifold  of dimension $n-1$. Let $S^i_j$ be sets of price vectors $p$ in $M$ such that housemate $i$ likes room $j$ at these prices. Consider the following conditions:

\medskip 

(A1) $\mathcal S^i=\{S^i_1,\ldots,S^i_n\}$, $i=1,\ldots,n$, is a cover of  $M$. 

\medskip

(A2)  $\deg{\mathcal C}\ne 0$, where  $ \mathcal C:=\mathcal F|_{\partial M}$,  $\mathcal F:=\{F_1,\ldots,F_n\}$, and $F_j:=\bigcup_{i=1}^n{S^i_j}$. 

\medskip 

The following theorem is equivalent to Corollary \ref{corM}.

  \begin{theorem} Suppose $n$ housemates in an $n$--bedroom house seek to decide who gets which room and for what part of the total rent. Also, suppose that the  conditions (A1) and (A2) hold. Then there exists a partition of the rent so that each person prefers a different room.
  \end{theorem}

\section{KKMS type theorems with boundary conditions}


Let us  extend definitions from Section 2 for any set of points (vectors)   $V:=\{v_1,\ldots, v_m\}$   in ${\Bbb R}^{n}$.
Denote by $c_V$ the center of mass of $V$,  $c_V:=(v_1+\ldots+v_m)/m. $

Let $\mathcal U=\{U_1,\ldots,U_m\}$ be  an open cover of a space $T$ and  $\Phi=\{\varphi_1,\ldots,\varphi_m\}$ be a partition of unity subordinate to $\mathcal U$.  
Let 
$$
\rho_{\mathcal U,\Phi,V}(x):=\sum\limits_{i=1}^m{\varphi_i(x)v_i}. 
$$

Suppose  $c_V$ lies outside of the image $\rho_{\mathcal U,\Phi,V}(T)$ in ${\Bbb R}^{n}$. Let for all $x\in T$
$$
f_{\mathcal U,\Phi,V}(x):=\frac{\rho_{\mathcal U,\Phi,V}(x)-c_V}{||\rho_{\mathcal U,\Phi,V}(x)-c_V||}. 
$$
Then $f_{\mathcal U,\Phi,V}$ is a continuous map from $T$ to ${\Bbb S}^{n-1}$. 

In \cite[Lemmas 2.1 and 2.2]{MusH} we proved that the homotopy class $[f_{\mathcal U,\Phi,V}]\in [T,{\Bbb S}^{n-1}]$ does not depend on $\Phi$ and  then  the  homotopy class $[f_{\mathcal U,V}]$ in $[T,{\Bbb S}^{n-1}]$ is well define.

\medskip

\noindent {\bf Notation 4.1:} Denote the homotopy class $[f_{\mathcal U,V}]$ in $[T,{\Bbb S}^{n-1}]$ by $\mu(\mathcal U,V)$. 

\medskip

Note that for the case $V=\vrt(\Delta^{n})$ we have
$
\mu(\mathcal U,V)=\mu(\mathcal U).
$

\medskip

 In  \cite[Lemma 2.4]{MusH} we show that this invariant is well defined also for closed covers. As above 
  we call a family of sets $\mathcal S=\{S_1,\ldots,S_m\}$  as a cover of a space $T$ if $\mathcal S$ is either an open or closed cover of $T$.


\begin{df} Let $I$ be a set of labels of cardinality $m$. Let  $V:=\{v_i,\, i\in I\}$,  be a set of points in ${\Bbb R}^{n}$.
	Then a nonempty subset $\mathcal{B}\subset I$ is said to be balanced with respect to $V$  if for all $i\in\mathcal{B}$ there exist non-negative $\lambda_i$ such that
$$
\sum\limits_{i\in\mathcal{B}}{\lambda_iv_i}=c_V, \; \mbox{ where } \;  \sum\limits_{i\in\mathcal{B}}{\lambda_i}=1 \, \mbox{ and } \, c_V:=\frac{1}{m}\sum\limits_{i\in I}{v_i}. 
$$
In other words, $c_V\in\conv\{v_i, \, i\in\mathcal{B}\}$, where $\conv(Y)$ denote the convex hull of $Y$ in ${\Bbb R}^{n}$.  
\end{df}


First we consider an extension of Theorems \ref{thKKM1}.

\begin{theorem} \label{th3}  
Let $V:=\{v_1,\ldots, v_m\}$  be a set of points in ${\Bbb R}^{n}$.
  Let $A$ be a subspace of a space $X$. Let  $(X,A)\in\ep_{n-1}$.  Let  $\mathcal F=\{F_1,\ldots,F_m\}$ be a cover of $(X,A)$. 
Suppose  $\mathcal S:=\mathcal F|_A$ is not null--homotopic.  Then   there is a balanced subset $\mathcal{B}$ in $I_m$ with respect to $V$ such that 
$$
\bigcap\limits_{i\in\mathcal{B}} {F_i}\ne\emptyset.
$$ 
\end{theorem}
\begin{proof} 
Assume the converse. Then there are no balanced subsets $\mathcal{B}$ in $I_m$ such that $\{F_i,\,i\in\mathcal{B}\}$ have a common point.
It implies that $c_V\notin\rho_{\mathcal F,V}(X)$ and therefore, $f_{\mathcal F,V}:X\to {\Bbb S}^{n-1}$ is well defined. On the other side, it is an extension of the map $f_{\mathcal S,V}:A\to {\Bbb S}^{n-1}$ with $[f_{\mathcal S,V}]\ne0$, a contradiction.  	
\end{proof}

\noindent{\bf Remark.} The assumption: ``$c_V\in\rho_{\mathcal S,V}(A)$ in ${\Bbb R}^{n}$'' is equivalent to the assumption: ``there is a balanced $\mathcal{B}$ in $I_m$  with respect to $V$ such that the intersection of all $S_i, \, {i\in\mathcal{B}},$ is not empty.'' Thus, if $c_V\in\rho_{\mathcal S,V}(A)$ or $c_V\notin\rho_{\mathcal S,V}(A)$ and $\mu(\mathcal S,V)\ne0$, then the intersection of the $F_i, \,  {i\in\mathcal{B}}$, is not empty. 

\medskip



Theorem \ref{th3} and \cite[Theorem 2.3]{MusH} imply the following extension of Corollary \ref{KKM1}.

\begin{cor} \label{th1} Let  $V:=\{v_1,\ldots, v_m\}$  be a set of points in ${\Bbb R}^{n}$. Let  $\mathcal F=\{F_1,\ldots,F_m\}$ be a cover of  ${\Bbb B}^{k}$  that  is not null--homotopic on the boundary. Then there is a balanced with respect to $V$  subset $\mathcal{B}$ in $I_m:=\{1,\ldots,m\}$ such that  the intersection of all $F_i, \, {i\in\mathcal{B}},$ is not empty.
\end{cor}

If $V=\vrt(\Delta^n)$, then this corollary implies the KKM theorem. It is also implies the KKMS theorem.

\begin{cor} (KKMS theorem \cite{Sh}) Let $\mathcal K$ be the collection of all non-empty subsets of  $I_{k+1}$. Consider a simplex $S$  in ${\Bbb R}^{k}$ with vertices $x_1,\ldots,x_{k+1}$.  Let $V:=\{v_\sigma, \, \sigma\in \mathcal K\}\subset {\Bbb R}^{k}$, where  $v_\sigma$ denotes the center of mass of $S_\sigma:=\{x_i, \, i\in\sigma\}$.
	
	Let $\mathcal C:= \{C_\sigma, \, \sigma\in\mathcal K\}$ be a cover of $|\Delta^{k}|$ such that for every $J\subset I_{k+1}$ the simplex $\Delta_J$ that is spanned by vertices from $J$ is covered by $\{C_\sigma, \, \sigma\in J\}$. Then there exists a balanced collection $\mathcal{B}$ in $\mathcal K$ with respect to $V$ such that 
$$
\bigcap\limits_{\sigma\in\mathcal{B}} {C_\sigma}\ne\emptyset.
$$
\end{cor}
\begin{proof} 
The assumptions of the corollary imply that $\mu(\mathcal C,V)=\deg(f_{\mathcal C,V})=1$.    Thus,  Corollary \ref{th1} yields the corollary.
\end{proof}


There are many extensions of the Sperner and KKM lemmas, in particular, that are Tucker's and Ky Fan's lemmas \cite{Bacon,DeLPS,Mus,MusSpT,MusS,MusH}. Actually, these lemmas can be derived from Theorem \ref{th1} using certain sets $V$. Let us consider as an example a generalization of Tucker's lemma. 

\begin{cor}  Let $V:=\{\pm e_1,\ldots,\pm e_n\}$, where $e_1,\ldots,e_n$ is a basis in ${\Bbb R}^n$. 	
 Let   $\mathcal F=\{F_1,F_{-1}\ldots,F_n,F_{-n}\}$ be a  cover of  ${\Bbb B}^{k}$  that  is not null--homotopic on the boundary. Then there is $i$ such that the intersection of $F_i$ and $F_{-i}$ is not empty.

In particular, if $\mathcal F$ is antipodally symmetric on the boundary of ${\Bbb B}^{k}$, i. e. for all $i$ we have $S_{-i}=-S_i,$ where $S_j:=F_j|_{{\Bbb S}^{k-1}}$, then there is $i$ such that  $F_i\cap F_{-i}\ne\emptyset$. 	
\end{cor}
\begin{proof} Note that any balanced subset with respect to $V$ consists of pairs $(i,-i)$, $i=1,\ldots,n$. It yields the first part of the theorem. 
	
By assumption, $\mathcal S$ is antipodally symmetric on ${\Bbb S}^{k-1}$. Then $\rho_{\mathcal S,V}:{\Bbb S}^{k-1}\to {\Bbb R}^{n}$ is an odd (antipodal) map. If $k>n$, then the Borsuk--Ulam theorem implies there $x\in{\Bbb S}^{k-1}$ such that $\rho_{\mathcal S,V}(x)=0$. Therefore, there is $i$ such that  $S_i\cap S_{-i}\ne\emptyset$. 

If for all $i$ we have  $S_i\cap S_{-i}=\emptyset$, then $k\le n$. Then the odd mapping theorem (see \cite{MusS}) implies that $k=n$ and $\deg(f_{\mathcal S,V})$ is odd. Thus, $\mu({\mathcal S,V})\ne0$ and from Theorem \ref{th1} follows the second part of the corollary. 
\end{proof}

\begin{df} Let  $V:=\{v_1,\ldots, v_m\}$  be a set of points in ${\Bbb R}^{n}$. Let $T$ be a triangulation of an $n$-dimensional manifold $M$. Let $L:\vrt(T)\to I_m$ be a labeling. 
We say that a simplex $s\in T$ is BL  {(Balanced Labelled)} if the set of labels $L(s):=\{L(v), v\in\vrt(s)\}$ is balanced with respect to $V$
\end{df}

\cite[Theorems 3.3 and 3.5]{MusH} yield the following theorem. 

\begin{theorem} \label{th2} Let  $V:=\{v_1,\ldots, v_m\}$  be a set of points in ${\Bbb R}^{n}$. 
Let  $T$ be a triangulation of an oriented manifold $M$  of dimension $n$ with boundary. Let $L:\vrt(T)\to\{1,2,\ldots,m\}$ be a labeling such that $T$ has no BL--simplices on the boundary. Then $T$ must contain at least $|\deg(L,\partial T)|$ distinct (internal)  BL--simplices. 
\end{theorem}

\begin{cor} \label{corS}
Let $T$ be a triangulation of a compact oriented PL--manifold $M$ of dimension $(m-1)$ with boundary. Then for any labeling $L:\vrt(T)\to \{1,2,\ldots,m\}$ the triangulation $T$ must  contain at least $|\deg(L,\partial T)|$ fully colored $(m-1)$--simplices. 	
\end{cor}
	\begin{proof}  Let   $V:=\{v_1,\ldots,v_m\}$  be the set of vertices of an $(m-1)$--simplex in ${\Bbb R}^{m-1}$. Then $I_m$ is the only balanced subset in $I_m$ with respect to $V$.
 Thus,  Theorem \ref{th2} yields the corollary. 
	\end{proof}

\begin{cor} \label{corT} Let $T$ be a triangulation of  ${\mathbb B}^n$ that antipodally symmetric on the boundary. Let $L:\vrt(T)\to \{+1,-1,\ldots, +n,-n\}$ be a labelling  that is antipodal on the boundary.
	 Suppose there are no complementary edges on the boundary. Then $\deg(L,\partial T)$ is an odd integer and there are at least $|\deg(L,\partial T)|$ internal  complementary edges.
	 
(An edge in $T$ is called {complementary} if its two vertices are labelled by opposite numbers.)
\end{cor}
\begin{proof} Let $V:=\{\pm e_1,\ldots,\pm e_n\}$, where $e_1,\ldots,e_n$ is an orthonormal basis in ${\Bbb R}^n$. 	Then any balanced subset with respect to $V$ consists of pairs $(i,-i)$, $i=1,\ldots,n$.  The fact that $\deg(L,\partial T)$ is odd follows from the odd mapping theorem, see \cite{Mus}. Thus, Theorem \ref{th2} completes the proof. 
\end{proof}

\medskip

\begin{figure}
\begin{center}

  \includegraphics[clip,scale=0.9]{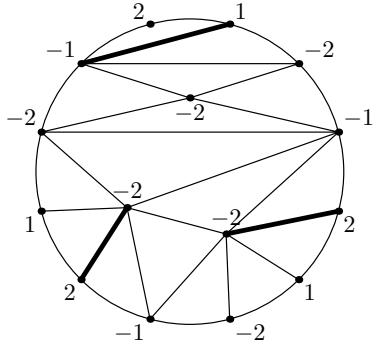}
\end{center}
\caption{Since $\deg(L,\partial T)=3$,  there are  three  complementary edges.}
\end{figure}

\medskip

  In Fig.2 is shown a labeling with  $\deg(L,\partial T)=3$. Then Corollary \ref{corT} yield that there are at least three complimentary edges.



\medskip


 \medskip

\noindent O. R. Musin\\ 
 University of Texas Rio Grande Valley, One West University Boulevard, Brownsville, TX, 78520 \\
{\it E-mail address:} oleg.musin@utrgv.edu

\end{document}